\documentclass[12pt]{amsart}

\usepackage{amssymb}
\usepackage{bm}
\usepackage{graphicx}
\usepackage{psfrag}
\usepackage{color}
\usepackage{url}
\usepackage{algpseudocode}
\usepackage{fancyhdr}
\usepackage{xy}
\usepackage{mathtools}
\usepackage[linesnumbered, ruled, vlined]{algorithm2e}
\usepackage{dirtytalk}
\usepackage{float}
\usepackage{stmaryrd}
\input xy
\xyoption{all}

\newtheorem*{maintheorem*}{Main Theorem}
\newtheorem{theorem}{Theorem}[section]
\newtheorem{prop}[theorem]{Proposition}

\newtheorem{lemma}[theorem]{Lemma}
\newtheorem{cor}[theorem]{Corollary}

\theoremstyle{definition}
\newtheorem{definition}[theorem]{Definition}
\newtheorem{remark}[theorem]{Remark}
\newtheorem{example}[theorem]{Example}
\numberwithin{equation}{section}

\newcommand{\nn}{\mathbb{N}}
\newcommand{\qq}{\mathbb{Q}}


\begin{document}
	
	\mbox{}
	\title{On the atomic structure of exponential Puiseux monoids and semirings}

	\author[1]{Sof\'ia Albizu-Campos}
	\address{Facultad de Matem\'atica y Computaci\'on, Universidad de La Habana, San L\'azaro y ~L, Vedado, Habana 4, CP-10400, Cuba}
	\email{sofiaalbizucampos@gmail.com}
	
	\author[2]{Juliet Bringas}
	\address{Facultad de Matem\'atica y Computaci\'on, Universidad de La Habana, San L\'azaro y~L, Vedado, Habana 4, CP-10400, Cuba}
	\email{julybm01@gmail.com}
	
	\author[3]{Harold Polo}
	\address{Department of Mathematics\\University of Florida\\Gainesville, FL 32611, USA}
	\email{haroldpolo@ufl.edu}
	
	\date{\today}

	\begin{abstract}
		We say that a Puiseux monoid is exponential provided that it is generated by some of the powers of a rational number. Here we study the atomic properties of exponential Puiseux monoids and semirings. First, we characterize atomic exponential Puiseux monoids, and we prove that the finite factorization property, the bounded factorization property, and the ACCP coincide in this context. Then we proceed to offer necessary and sufficient conditions for an exponential Puiseux monoid to satisfy the ACCP. We conclude by describing the exponential Puiseux monoids that are semirings. 
	\end{abstract}
	
	\maketitle

%%%%%%%%%%%%%%%%%%%%%%%%
\section{Introduction} \label{sec:intro}

An integral domain $R$ is atomic provided that every nonzero nonunit element is a product of finitely many atoms (i.e., irreducibles) of $R$, and it satisfies the ACCP if for each sequence $(x_n)_{n \geq 1}$ in $R$ satisfying that $x_{n + 1}$ divides $x_n$ in $R$ for each $n \in\nn$, there exists $k \in \nn$ such that $x_n$ and $x_k$ are associates for each $n \geq k$. Clearly, an integral domain $R$ satisfying the ACCP is also atomic. Indeed, take a nonzero nonunit element $x_0 \in R$. Then either $x_0$ is an atom or it can be written as $x_0 = x_1x_2$, where neither $x_1$ nor $x_2$ are units of $R$. If both elements $x_1$ and $x_2$ are atoms then we stop; otherwise, either $x_1$ or $x_2$ is not an atom of $R$, and we can repeat the process by factoring that element. The fact that $R$ satisfies the ACCP forces this procedure to stop, which means that every nonzero nonunit element is a product of atoms of $R$. 

In 1974, A. Grams disproved P. Cohn's assumption that every atomic integral domain satisfies the ACCP~\cite{Grams}. Since then, several papers have studied the interplay between the atomic property and the ACCP in the context of integral domains (see, for instance,~\cite{DDDFMZ1990,Zaks1982}). We can investigate these properties not only for integral domains but in the more general context of commutative and cancellative monoids, where it is not hard to find examples of atomic monoids that do not satisfy the ACCP (see, for example,~\cite[Corollary 5.5]{CGG2019}). Additionally, atomic monoids not satisfying the ACCP play an important role in understanding when being atomic is transferred from a commutative monoid $M$ to its monoid ring $R[M]$, a question posed by R. Gilmer in~\cite[page 189]{Gilmer} and partially answered by M. Roitman in~\cite{Roitman} and J. Coykendall and F. Gotti in~\cite{JCFG2019}.

In this article, we study the atomic structure of exponential Puiseux monoids, that is, additive submonoids of the nonnegative cone of rational numbers $\qq$ generated by some of the powers of a positive rational number. These monoids are a generalization of rational cyclic monoids (which are also semirings), which were introduced in~\cite{GG17} and deeper studied in~\cite{CGG19}. The class of exponential Puiseux monoids provides a fertile ground to study the interplay between the atomic property and the ACCP: this class consists mostly of atomic monoids and contains a large subclass of atomic monoids that do not satisfy the ACCP. 

We begin the next section by introducing not only the necessary background but also the notation we shall be using throughout this paper. In Section~3, we characterize the exponential Puiseux monoids that are atomic and show that the finite factorization property, the bounded factorization property, and the ACCP agree in the context of exponential Puiseux monoids. In Section~4, we provide a necessary condition and a sufficient condition for an exponential Puiseux monoid to satisfy the ACCP. We conclude by describing, in Section~5, the exponential Puiseux monoids that are semirings.

\section{Background} \label{sec:background}

In this section we introduce the concepts and notation related to our exposition. Reference material on non-unique factorization theory can be found in the monograph~\cite{GH06b} by A. Geroldinger and F. Halter-Koch. 

\subsection{Notation}

We let $\nn$ and $\nn_0$ denote the set of positive and nonnegative integers, respectively, and let $\mathbb{P}$ denote the set of prime numbers. In addition, if $X$ is a subset of the rational numbers then we set $X_{<q} \coloneqq \{x \in X \mid 0 \leq x < q\}$. In the same way we define $X_{\leq q}$, $X_{>q}$ and $X_{\geq q}$. Additionally, if $n \in \nn$ and $X \subseteq \nn_0$ then we set $X - n \coloneqq \{x - n \mid x \in X_{\geq n}\}$. For a rational number $r = n/d$ with $n$ and $d$ relatively prime positive integers, we call $n$ the \emph{numerator} and $d$ the \emph{denominator} of $r$, and we set $\mathsf{n}(r) := n$ and $\mathsf{d}(r) := d$. For $k,m$ nonnegative integers such that $k \leq m$, we denote by $\llbracket k,m \rrbracket$ the set of integers between $k$ and $m$, i.e., $\llbracket k,m \rrbracket \coloneqq \left\{ s \in \nn_0 \mid k \leq s \leq m \right\}$.

\subsection{Puiseux monoids}

A monoid is defined to be a semigroup with identity, and we tacitly assume that every monoid we refer to here is cancellative, commutative, and reduced (i.e., its only invertible element is the identity). Unless we specify otherwise, we shall use additive notation for monoids. Now let $M$ be a monoid. We denote by $\mathcal{A}(M)$ the set consisting of those elements $a \in M^{\bullet} \coloneqq M\setminus\{0\}$ such that if $a = x + y$ for some $x,y \in M$ then either $x = 0$ or $y = 0$. The elements of $\mathcal{A}(M)$ are called \emph{atoms}. For a subset $S \subseteq M$, we denote by $\langle S \rangle$ the minimal submonoid of $M$ containing $S$, and if $M = \langle S \rangle$ then we say that $S$ is a \emph{generating set} of $M$. We call the monoid $M$ \emph{atomic} provided that $M = \langle\mathcal{A}(M)\rangle$. On the other hand, it is said that $M$ is \emph{antimatter} provided that $\mathcal{A}(M) = \emptyset$. For $x,y \in M$, it is said that $x$ \emph{divides} $y$ \textit{in} $M$ if there exists $x' \in M$ such that $y = x + x'$ in which case we write $x \,|_M \,y$ and drop the subscript $M$ whenever $M = (\nn, \cdot)$. %As usual, we use the notation $x \,|\, y$ to indicate that $x,y \in\zz^{\bullet}$ and $x \,|_{(\zz^{\bullet},\times)}\, y$. 
A subset $I$ of $M$ is an \emph{ideal} of $M$ on condition that $I + M \subseteq I$. An ideal $I$ is \emph{principal} if $I = x + M$ for some $x \in M$. Furthermore, it is said that $M$ satisfies the \emph{ascending chain condition on principal ideals} (or \emph{ACCP}) if every increasing sequence of principal ideals of $M$ eventually stabilizes. If $M$ satisfies the ACCP then it is atomic~\cite[Proposition 1.1.4]{GH06b}.

A \emph{numerical monoid} $N$ is an additive submonoid of $\nn_0$ whose complement in $\nn_0$ is finite. If $N \neq \nn_0$ then the greatest integer that is not an element of $N$ is called the \emph{Frobenius number} of $N$ and is denoted by $F(N)$. It is well known that numerical monoids are always finitely generated and, therefore, atomic. An introduction to numerical monoids can be found in~\cite{GSJCR2009}. A \emph{Puiseux monoid} is an additive submonoid of $\qq_{\ge 0}$. Clearly, Puiseux monoids are a natural generalization of numerical monoids. The atomic structure of Puiseux monoids has received considerable attention during the past few years (see, for instance,~\cite{fG16,fG17b,GG17}). In particular, some authors have studied rational cyclic semirings, i.e., Puiseux monoids generated by the elements of a finite geometric progression (see~\cite{CGG19,GG17,HP2020}). Unlike Puiseux monoids in general, rational cyclic semirings have a tractable atomic structure, and this allows to nicely compute some of their factorization invariants (for some of these computations, see~\cite{CGG19}).

\begin{definition}
	The \emph{rational cyclic semiring} over $r \in \qq_{>0}$ is the Puiseux monoid generated by the nonnegative powers of $r$, i.e., $S_r = \langle r^n \mid n\in\nn_0\rangle$. 
\end{definition}

Most rational cyclic semirings are atomic as the next theorem indicates.

\begin{theorem}\cite[Theorem 6.2]{GG17} \label{theorem: cyclic rational semirings are atomic unless generated by unit fraction}
	Let $r \in \qq_{>0}$ and $S_r = \langle r^n \mid n \in \nn_0 \rangle$. The following statements hold.
	\begin{enumerate}
		\item If $\mathsf{d}(r) = 1$ then $S_r$ is atomic with $\mathcal{A}(S_r) = \{1\}$.
		\item If $\mathsf{d}(r) > 1$ and $\mathsf{n}(r) = 1$ then $S_r$ is antimatter.
		\item If $\mathsf{d}(r) > 1$ and $\mathsf{n}(r) > 1$ then $S_r$ is atomic with $\mathcal{A}(S_r) = \{r^n \mid n \in \nn_0\}$.
	\end{enumerate} 
\end{theorem}

\subsection{Factorizations}

Let $M$ be a commutative, cancellative, reduced, and atomic monoid. The \emph{factorization monoid} of $M$, denoted by $\mathsf{Z}(M)$, is the free commutative monoid on $\mathcal{A}(M)$. The elements of $\mathsf{Z}(M)$ are called \emph{factorizations}, and if $z = a_1 + \cdots + a_n \in \mathsf{Z}(M)$ for $a_1, \ldots, a_n \in\mathcal{A}(M)$ then we say that the length of $z$, denoted by $|z|$, is $n$. %Throughout this paper, whenever we refer to an expression of the form $z = \sum_{i = 0}^{K} c_i r^{s_i}$, where $c_i, s_i, K \in\nn_0$ for each $i \in\llbracket 0,K \rrbracket$, we assume that $s_i < s_j$ for $i < j$. 
The unique monoid homomorphism $\pi\colon\mathsf{Z}(M) \to M$ satisfying that $\pi(a) = a$ for all $a \in\mathcal{A}(M)$ is called the \emph{factorization homomorphism} of $M$. For each $x \in M$, there are two important sets associated to $x$:
\[
\mathsf{Z}_M(x) \coloneqq \pi^{-1}(x) \subseteq \mathsf{Z}(M) \hspace{0.6 cm}\text{ and } \hspace{0.6 cm}\mathsf{L}_M(x) \coloneqq \{|z| : z \in\mathsf{Z}_M(x)\},
\]
which are called the \emph{set of factorizations} of $x$ and the \emph{set of lengths} of $x$, respectively; if the monoid $M$ is clear from the context then we drop the subscript. In addition, the collection $\mathcal{L}(M) \coloneqq \{\mathsf{L}(x) \mid x \in M\}$ is called the \emph{system of sets of lengths} of $M$. See~\cite{aG16} for a survey on sets of lengths. We say that $M$ satisfies the \emph{finite factorization property} (or \emph{FFP}) provided that $\mathsf{Z}(x)$ is nonempty and finite for all $x \in M$. In this case we also say that $M$ is an \emph{FFM}. Similarly, $M$ satisfies the \emph{bounded factorization property} (or \emph{BFP}) if $\mathsf{L}(x)$ is nonempty and finite for all $x \in M$, and in this case we say that $M$ is a \emph{BFM}. It is clear that each FFM is a BFM, and it follows from~\cite[Corollary~1.3.3]{GH06b} that each BFM satisfies the ACCP.

\section{Atomic Structure of Exponential Puiseux Monoids}

In this section we investigate the atomic structure of \emph{exponential Puiseux monoids}. Specifically, we show that most exponential Puiseux monoids are atomic, and we prove that the FFP, the BFP, and the ACCP are equivalent in this context. But first let us define the object of study of this paper.

%\begin{definition} \label{def:exponential PM}
%	Take $r \in \qq_{>0}$, and let $S$ be a subset of $\nn_0$ containing $0$. We let $M_{r,S}$ denote the Puiseux monoid
%	\[
%	M_{r,S} \coloneqq \langle r^s \mid s \in S \rangle,
%	\]
%	which we called \emph{exponential}. In addition, if $S = \{s_n \mid n \in \nn_0\}$ with $s_0 \leq \cdots \leq s_n < s_{n + 1} < \cdots$ then we set $\delta_n \coloneqq s_{n+1} - s_n$ for each $n \in \nn_0$.
%\end{definition}

\begin{definition} \label{def: exponential PM}
	Take $r \in \qq_{>0}$ and let $S = \{s_0 = 0 < s_1 < s_2 < \cdots\}$ be a subset of $\nn_0$. We let $M_{r,S}$ denote the Puiseux monoid
	\[
	M_{r,S} \coloneqq \langle r^{s_n} \mid n \in \nn_0 \rangle,
	\]
	which we called \emph{exponential}. 
\end{definition}

%\begin{remark}
%	With notation as in Definition~\ref{def: exponential PM}, we denote by $s_n$ the $(n + 1)$th smallest element of $S$ and set $\delta_n \coloneqq s_{n+1} - s_n$. 
%\end{remark}

\begin{remark} \label{rem: identity relating delta_i and s_m}
		With notation as in Definition~\ref{def: exponential PM} and for each $n \in \nn_0$, we denote by $s_n$ the $(n + 1)$th smallest element of $S$ and set $\delta_n \coloneqq s_{n+1} - s_n$. Then it is easy to see that the equality $\sum_{i = 0}^{m - 1} \delta_i = s_m$ holds.
\end{remark}

\begin{remark}
	Finitely generated Puiseux monoids are isomorphic to numerical monoids by ~\cite[Proposition 3.2]{fG17a}. As it is clear that numerical monoids satisfy the finite factorization property, we assume that exponential Puiseux monoids are not finitely generated unless we specify otherwise. %Moreover, if $s$ is the $n$th smallest element of $S$, we sometimes refer to $s$ as $s_n$. 
\end{remark}

%\begin{remark}
%	Whenever we refer to an expression of the form $\sum_{i=0}^n c_i r^{s_i}$ in this article, where $n, c_i, s_i \in \nn_0$ for every $i \in \llbracket 0, n \rrbracket$, we tacitly assume that $s_i < s_j$ for $i < j$.
%\end{remark}

In the literature we can find many instances in which particular families of exponential Puiseux monoids have been studied. Consider the following examples.

\begin{example}
	Let $N = (\nn_0, +)$. Clearly, $N$ is an exponential Puiseux monoid. Furthermore, $\nn_0 \subseteq M$ for all exponential Puiseux monoids $M$. As mentioned before, the atomic structure of $N$ is not hard to describe: it is an FFM. 
\end{example}

\begin{example}
	Let $r$ be a positive rational number and consider the Puiseux monoid generated by the set $\{r^n \mid n\in\nn_0\}$. These monoids, introduced in~\cite{GG17}, are called rational cyclic semirings because they are also closed under multiplication. The atomic structure of rational cyclic semirings is not very rich since they are almost always atomic (Theorem~\ref{theorem: cyclic rational semirings are atomic unless generated by unit fraction}) and satisfy the ACCP if and only if $r \geq 1$~(\cite[Corollary 4.4]{CGG2019}, Theorem~\ref{theorem: cyclic rational semirings are atomic unless generated by unit fraction}, and \cite[Theorem 5.6]{fG16}).
\end{example}

\begin{example}
	Let $\mathcal{B}$ be a finite subset of $\qq_{>0}$ and set $M_{\mathcal{B}} \coloneqq \langle b^n \mid b \in \mathcal{B},\, n \in \mathbb{N}_0 \rangle$. We say that $M_{\mathcal{B}}$ is the rational multicyclic monoid over $\mathcal{B}$ provided that $\mathcal{B}$ is minimal, that is, if $\mathcal{B}' \subsetneq \mathcal{B}$ then $M_{\mathcal{B}'} \subsetneq M_{\mathcal{B}}$. These monoids are a direct generalization of rational cyclic semirings, and they were studied by the third author in~\cite{HP2020}. Many rational multicyclic monoids are exponential. Indeed, it is not hard to see that $M_{\mathcal{B}}$ is an exponential Puiseux monoid if and only if the elements of $\mathcal{B}$ are powers of the same positive rational number $r$. The atomic structure of rational multicyclic monoids is somewhat similar to that one of rational cyclic semirings (see~\cite[Theorem 3.7]{HP2020}).
\end{example}

As is the case for rational cyclic semirings, it is straightforward to describe the exponential Puiseux monoids that are atomic. 

\begin{prop} \label{prop:atomicity of exponential PMs}
	Let $M_{r,S}$ be an exponential Puiseux monoid. Then the following statements hold. 
	\begin{enumerate}
        \item If $\mathsf{d}(r)=1$ then $M_{r,S} = \nn_0$ and so $\mathcal{A}(M_{r,S}) = \{1\}$.
        \item If $\mathsf{n}(r)=1$ and $\mathsf{d}(r)>1$ then $M_{r,S}$ is antimatter and so $\mathcal{A}(M_{r,S}) = \emptyset$.
		\item If $\mathsf{n}(r)>1$ and $\mathsf{d}(r)>1$ then $M_{r,S}$ is atomic and $\mathcal{A}(M_{r,S}) = \{r^s \mid s \in S\}$.
	\end{enumerate}
\end{prop}

\begin{proof}
	It is easy to see that if $\mathsf{d}(r) = 1$ then $M_{r,S} = \langle 1 \rangle = \nn_0$. On the other hand, the equation $\mathsf{d}(r)^{-s_n} = \mathsf{d}(r)^{\delta_n}\mathsf{d}(r)^{-s_{n + 1}}$ holds for all $n \in \nn_0$ from which $(2)$ follows since $\mathsf{n}(r) = 1$. As for~$(3)$, if $r^s \not\in\mathcal{A}(M_{r,S})$ for some $s \in S$ then $r^s \not\in\mathcal{A}(S_r)$, where $S_r= \langle r^n \mid n\in\nn_0\rangle$, contradicting Theorem~\ref{theorem: cyclic rational semirings are atomic unless generated by unit fraction}.
\end{proof}

As we indicated above, a rational cyclic semiring $S_r$ with $r < 1$ does not satisfy the ACCP by~\cite[Corollary 4.4]{CGG2019} and Theorem~\ref{theorem: cyclic rational semirings are atomic unless generated by unit fraction}. On the other hand, if $r \geq 1$ then $S_r$ is an FFM by virtue of~\cite[Theorem 5.6]{fG16}. Consequently, the FFP, the BFP, and the ACCP coincide in the context of rational cyclic semirings. As we show next, this result can be extended to exponential Puiseux monoids. First we collect two technical lemmas.

\begin{lemma} \label{Lemma 2.2}
	Let $x$ be a nonzero element of an atomic exponential Puiseux monoid $M_{r,S}$ with $r \in\qq_{<1}$ and consider a factorization $z=\sum_{i=0} ^n c_i r^{s_i} \in\mathsf{Z}(x)$ with coefficients $c_0, \ldots, c_n \in \nn_0$. The following conditions hold.
	\begin{enumerate}
		\item $\min \mathsf{L}(x) =  |z|$ if and only if $c_i < \mathsf{d}(r)^{\delta_{i-1}}$ for each $i \in \llbracket 1,n \rrbracket$. 
		\item There exists exactly one factorization $z_0$ in $\mathsf{Z}(x)$ of minimum length.
		\item $\max \mathsf{L}(x) = |z|$ if and only if $c_i < \mathsf{n}(r)^{\delta_i}$ for all $i\in\llbracket 0,n \rrbracket$.
		\item There exists, at most, one factorization of maximum length of $x$.
		\item If $c_i < \mathsf{n}(r)$ for each $i \in \llbracket 0,n \rrbracket$ then $| \mathsf{Z}(x)| =1$.
	\end{enumerate}
\end{lemma}

\begin{proof}
	The proofs of $(1)$ and $(2)$ are left to the reader as they follow the proof of~\cite[Lemma~3.1]{CGG19}. To prove the direct implication of $(3)$ note that if $c_i \geq \mathsf{n}(r)^{\delta_i}$ for some $i \in\llbracket 0,n \rrbracket$ then by using the transformation $\mathsf{n}(r)^{\delta_i}r^{s_i} = \mathsf{d}(r)^{\delta_i}r^{s_{i + 1}}$ we can generate a factorization $z^*\in\mathsf{Z}(x)$ such that $|z^*| > |z|$ since $\mathsf{d}(r) > \mathsf{n}(r)$. Conversely, consider a factorization $z' = \sum_{i=0}^t d_i r^{s_i} \in\mathsf{Z}(x)$ with coefficients $d_0, \ldots, d_t \in \nn_0$ and suppose, by way of contradiction, that $|z'| > |z|$. There is no loss in assuming that $n \leq t$. By applying the identity $\mathsf{n}(r)^{\delta_m}r^{s_m} = \mathsf{d}(r)^{\delta_m}r^{s_{m + 1}}$ with $m \in\nn_0$ finitely many times, we can generate factorizations $z' = z_1, \ldots, z_k = \sum_{i = 0}^{q} e_ir^{s_i} \in\mathsf{Z}(x)$ such that $|z_j| < |z_{j + 1}|$ for $j \in \llbracket 1,k - 1 \rrbracket$ and $e_i < \mathsf{n}(r)^{\delta_i}$ for $i \in\llbracket 0,t \rrbracket$. Note that $t \leq q$. Let $l \in\llbracket 0,n \rrbracket$ be the smallest index such that $c_l \neq e_l$. Note that such an index $l$ exists given that the inequalities $|z| < |z'| \leq |z_k|$ hold. This implies that $z$ and $z_k$ are two different factorizations of $x$. Thus,
	\[
	(c_l - e_l)r^{s_l} = \sum_{i = l + 1}^{n} (e_i - c_i)r^{s_i} + \sum_{i = n + 1}^{q} e_i r^{s_i}\!\!,
	\]
	which implies that $\mathsf{n}(r)^{\delta_l}\, | \,c_l - e_l$. This contradiction proves that our hypothesis is untenable. Consequently, $z$ is a factorization of maximum length of $x$. Note that~$(4)$ follows readily from~$(3)$ and $(5)$ is a direct consequence of~$(1)$ and~$(3)$. 
\end{proof}

\begin{lemma} \label{lem: infinite factorization generates nonACCP sequence}
	Let $M_{r,S}$ be as in Lemma~\ref{Lemma 2.2} and let $x = \pi(k_0\,\mathsf{d}(r)^{\delta_i}r^{s_{i + 1}})$ be an element of $M_{r,S}^{\bullet}$ for some $k_0, i \in\nn_0$. If $x$ does not have a factorization of maximum length then $x = y + x'$ with $y\in M_{r,S}^{\bullet}$ and $x' = \pi(k\,\mathsf{d}(r)^{\delta_j}r^{s_{j + 1}})$ for some $k, j \in\nn$. Furthermore, $x'$ does not have a factorization of maximum length.
\end{lemma}

\begin{proof}
	
	We start by describing a process that will generate a sequence of factorizations of $x$ each having the form $z_m = k_m \mathsf{d}(r)^h r^{s_{m + i + 1}}$ for some positive integers $k_m$ and $h$. Let $z_0 = k_0 \mathsf{d}(r)^{\delta_i}r^{s_{i + 1}}$. Assume that $z_j = k_j \mathsf{d}(r)^{h_j}r^{s_{i + j + 1}}$ was already defined for some $j \in \nn_0$. Now if $\mathsf{n}(r)^{\delta_{i + j + 1}} \nmid k_j$ then the process stops, and we obtain a sequence of factorizations $z_0, \ldots, z_j \in \mathsf{Z}(x)$. On the other hand, if $\mathsf{n}(r)^{\delta_{i + j + 1}} \mid k_j$ then by applying the transformation $\mathsf{n}(r)^{\delta_{i + j + 1}} r^{s_{i + j + 1}} = \mathsf{d}(r)^{\delta_{i + j + 1}} r^{s_{i + j + 2}}$ we generate a new factorization $z_{j + 1} = k_{j + 1} \mathsf{d}(r)^{h_{j + 1}} r^{s_{i + j + 2}} \in \mathsf{Z}(x)$ for some positive integers $k_{j + 1}$ and $h_{j + 1}$ such that $k_j > k_{j + 1}$. Since there is no infinite strictly decreasing sequence of positive integers, this process eventually stops, and it yields a sequence of factorizations $z_0, \ldots, z_n = k_n \mathsf{d}(r)^h r^{s_{n + i + 1}} \in \mathsf{Z}(x)$, where $\mathsf{n}(r)^{\delta_{n + i + 1}} \nmid k_n$. 
	
	Next note that since $x$ has no factorization of maximum length, the inequality $\mathsf{n}(r)^{\delta_{n + i + 1}} \leq k_n \mathsf{d}(r)^h$ holds by Lemma~\ref{Lemma 2.2} (part (3)). We have $k_n\mathsf{d}(r)^h = k' \mathsf{n}(r)^{\delta_{n + i + 1}} + l$ with $k' \in\nn$ and $l\in\llbracket 0,\mathsf{n}(r)^{\delta_{n + i + 1}} - 1 \rrbracket$. Since $\mathsf{n}(r)$ and $\mathsf{d}(r)$ are relatively prime numbers for each $r \in \qq_{>0}$, we have $l \neq 0$. Thus,
	\[
	k' \mathsf{d}(r)^{\delta_{n 
	+ i + 1}} \cdot r^{s_{n + i + 2}} + l\cdot r^{s_{n + i + 1}} \in\mathsf{Z}(x).
	\]
	Note that $z = l\, r^{s_{n + i + 1}}$ is a factorization of maximum length of $y = \pi(z)$ by Lemma~\ref{Lemma 2.2}. We conclude by proving that $x' = \pi(k' \mathsf{d}(r)^{\delta_{n + i + 1}} \cdot r^{s_{n + i + 2}})$ has no factorization of maximum length. Suppose, by way of contradiction, that $x'$ has a factorization of maximum length $z^* \in \mathsf{Z}(x')$. There exists a sequence of factorizations $z_0, \ldots, z_t = z^*$, where $z_0 = k' \mathsf{d}(r)^{\delta_{n + i + 1}} r^{s_{n + i + 2}}$. In fact, if $z_j$ is already defined for some $j \in \nn_0$ then either $z_j$ has a summand of the form $\mathsf{n}(r)^{\delta_k}r^{s_k}$ (with $k$ a nonnegative integer) in which case we generate a factorization $z_{j + 1}$ using the transformation $\mathsf{n}(r)^{\delta_k}r^{s_k} = \mathsf{d}(r)^{\delta_k}r^{s_{k + 1}}$ or $z_j$ does not have such a summand in which case $z_j = z^*$ by Lemma~\ref{Lemma 2.2} (parts (3) and (4)). The aforementioned replacements not only yield each time a factorization of bigger length, which means that we cannot carry out these transformations infinitely many times, but also increase the exponents of the summands involved. Consequently, $z^* + z$ is a factorization of maximum length of $x$ by Lemma~\ref{Lemma 2.2}. This contradiction proves that $x'$ has no factorization of maximum length.
%	 This, together with the fact that $x$ does not have a factorization of maximum length, implies that $x' = \pi(k' \mathsf{d}(r)^{\delta_{n + i + 1}} \cdot r^{s_{n + i + 2}})$ does not have a factorization of maximum length.
%	
%	When this occurs we obtain a factorization $z_n = k\mathsf{d}(r)^hr^{s_t} \in\mathsf{Z}(x)$, where $k,h,t \in\nn$, such that $\mathsf{n}(r)^{\delta_t} \leq K\mathsf{d}(r)^h$ and $\mathsf{n}(r)^{\delta_t} \nmid k$. Since $\gcd(\mathsf{n}(r),\mathsf{d}(r)) = 1$, we have $k\mathsf{d}(r)^h = k' \mathsf{n}(r)^{\delta_t} + l$ with $k' \in\nn$ and $l\in\llbracket 1,\mathsf{n}(r)^{\delta_t} - 1 \rrbracket$. Thus,
%	\[
%		l\cdot r^{s_t} + k' \mathsf{d}(r)^{\delta_t} \cdot r^{s_{t + 1}} \in\mathsf{Z}(x).
%	\]
%	It is not hard to see that $x' = \pi(k' \mathsf{d}(r)^{\delta_t} r^{s_{t + 1}})$ does not have a factorization of maximum length by Lemma~\ref{Lemma 2.2}, and $y = \pi(l r^{s_t})$ is an element of $M_{r,S}^{\bullet}$ which concludes our proof.
\end{proof}

Now we are in a position to prove the main result of this section.

\begin{theorem} \label{prop: ACCP implies FF}
	Let $M_{r,S}$ be an atomic exponential Puiseux monoid. Then the following statements are equivalent.
	\begin{enumerate}
		\item $M_{r,S}$ satisfies the FFP.
		\item $M_{r,S}$ satisfies the BFP.
		\item $M_{r,S}$ satisfies the ACCP.
	\end{enumerate}
\end{theorem}
\begin{proof}
	Recall that, for commutative and cancellative monoids, $(1)$ implies $(2)$ by definition and $(2)$ implies $(3)$ by \cite[Corollary~1.3.3]{GH06b}. Then our proof reduces to showing that $(3)$ implies $(1)$. If $r \geq 1$ then our result follows from~\cite[Theorem 5.6]{fG16}; consequently, we may assume $r < 1$. Assume that $M_{r,S}$ is not an FFM, and let $x \in M_{r,S}^{\bullet}$ with $|\mathsf{Z}(x)| = \infty$. 
	
	By way of contradiction, suppose that $x$ has a factorization of maximum length $z=\sum_{i=0}^m c_i r^{s_i} \in \mathsf{Z}(x)$ with coefficients $c_0, \ldots, c_m \in \nn_0$. Note that the set $S' = \{s \in S : r^{s}\,|_{M_{r,S}}\, x\}$ has infinite cardinality since $|\mathsf{Z}(x)| = \infty$. Clearly, there exists $s_j \in S'$ such that $s_j > s_m$. This implies that there exists $z_1 =\sum_{i=0}^l d_i r^{s_i} \in\mathsf{Z}(x)$, where $l, d_i \in\nn_0$, $j < l$, $d_j > 0$, and $s_i \in S$ for each $i \in\llbracket 0,l \rrbracket$. Since $d_j > 0$, we have $z_1 \neq z$. By virtue of Lemma~\ref{Lemma 2.2}, the inequality $d_k \geq \mathsf{n}(r)^{\delta_k}$ holds for some $k \in\llbracket 0,l \rrbracket$. Consequently, one can obtain a factorization $z_2$ from $z_1$ by applying the transformation $\mathsf{n}(r)^{\delta_k}r^{s_k} = \mathsf{d}(r)^{\delta_k}r^{s_{k + 1}}$\!. Note that $|z_2| > |z_1|$. Moreover, $z_2 \neq z$ as either $s_j$ or $s_{j + 1}$ shows up in the factorization $z_2$. Repeating this process for $z_2$ we can obtain a factorization $z_3 \in \mathsf{Z}(x)$ such that $|z_3| > |z_2| > |z_1|$ and $z_3 \neq z$. It follows by induction that there exists a sequence $z_1, z_2, \ldots$ of elements of $\mathsf{Z}(x)$ such that $|z_n| < |z_{n + 1}| < |z|$ for each $n \in\nn$, a contradiction. Hence $x$ has no factorization of maximum length.
	
	Now let $z=\sum_{i=0}^m c_i r^{s_i} \in\mathsf{Z}(x)$ with coefficients $c_0, \ldots, c_m \in \nn_0$. By applying the transformation $\mathsf{n}(r)^{\delta_i}r^{s_i} = \mathsf{d}(r)^{\delta_i}r^{s_{i + 1}}$ finitely many times over all summands $c_ir^{s_i}$ with $i \in\llbracket 0,m - 1 \rrbracket$ we can generate a factorization $z' = \sum_{i=0}^m d_i r^{s_i} \in\mathsf{Z}(x)$ with $d_i \in\nn_0$ for each $i \in\llbracket 0,m \rrbracket$ such that $d_i < \mathsf{n}(r)^{\delta_i}$ for each $i \in\llbracket 0, m - 1 \rrbracket$. Note that $d_m = h\, \mathsf{n}(r)^{\delta_m} + l$, where $h \in \nn$ and $l\in\llbracket 0,\mathsf{n}(r)^{\delta_m}  - 1 \rrbracket$; otherwise, $z'$ would be the factorization of maximum length of $x$ by Lemma~\ref{Lemma 2.2}, which is impossible. Then $x' = \pi(h \mathsf{n}(r)^{\delta_m}r^{s_m}) = \pi(h \mathsf{d}(r)^{\delta_m}r^{s_{m + 1}})$ has no factorization of maximum length. It is not hard to see that using Lemma~\ref{lem: infinite factorization generates nonACCP sequence} one can generate a sequence $(y_n)_{n \geq 1}$ of elements of $M_{r,S}$ such that $y_1 > y_2 > \cdots$ and $y_{n + 1}\, |_{M_{r,S}} \,y_n$ for all $n\in\nn$. Therefore, $M_{r,S}$ does not satisfy the ACCP.
\end{proof}

If no exponential Puiseux monoid $M_{r,S}$ with $r < 1$ satisfies the ACCP then Theorem~\ref{prop: ACCP implies FF} holds trivially by~\cite[Theorem 5.6]{fG16}. However, this is far from being the case as we will show in the next section.

\section{Exponential Puiseux Monoids and the ACCP}

Now we proceed to study the ACCP in the context of exponential Puiseux monoids. We show that there are infinitely many exponential Puiseux monoids that satisfy the ACCP and infinitely many that do not. In this section, we provide a necessary condition and a sufficient condition for an exponential Puiseux monoid to satisfy the ACCP. 

\begin{prop} \label{prop: necessary condition for satisfying the ACCP}
	Let $M_{r,S}$ be an atomic exponential Puiseux monoid with $r < 1$. If $M_{r,S}$ satisfies the ACCP then
	\begin{equation} \label{eq: limsup}
	 	\mathsf{d}(r) \leq \mathsf{n}(r)\cdot \limsup_{n \to\infty} \sqrt[s_n]{\mathsf{n}(r)^{\delta_n}}.
	\end{equation}
\end{prop}

\begin{proof}
	We prove that if $M_{r,S}$ satisfies the ACCP then 
		\begin{equation} \label{eq: series determining ACCP}
			\sum_{k = 0}^{\infty} \left(\mathsf{n}(r)^{\delta_{k}} - 1\right) r^{s_k}
		\end{equation}
	does not converge. For the sake of a contradiction, suppose that the series~\eqref{eq: series determining ACCP} (of positive terms) converges to a real number $R > 0$, and let $K\in\nn$ such that 
	\[
	K > R = r^{\delta_{-1}}\sum_{k = 0}^{\infty} \left(\mathsf{n}(r)^{\delta_{k}} - 1\right) r^{s_k}\!\!,
	\]
	where $\delta_{-1} = 0$. In addition, let $z_0 = K r^{s_0} \in\mathsf{Z}(K)$. We now prove that if $z_m = \sum_{i = 0}^{m} c_ir^{s_i}\in\mathsf{Z}(K)$ is a factorization of $K$ such that
	\[
	c_m > \prod_{j = -1}^{m - 1} \!\!r^{-\delta_j} \sum_{k = m}^{\infty} \left(\mathsf{n}(r)^{\delta_k} - 1\right)r^{s_k}
	\]
	then there exists a factorization $z_{m + 1} = \sum_{i = 0}^{m + 1} d_i r^{s_i} \in\mathsf{Z}(K)$, where
	\[
	d_{m + 1} > \prod_{j = -1}^{m}\!\! r^{-\delta_j} \sum_{k = m + 1}^{\infty} \left(\mathsf{n}(r)^{\delta_k} - 1\right)r^{s_k}\!\!.
	\]
	Note that $c_m > r^{s_m}\left(\mathsf{n}(r)^{\delta_m} - 1\right)\prod_{j = -1}^{m - 1} r^{-\delta_j}  = \mathsf{n}(r)^{\delta_m} - 1$, where the equality holds by Remark~\ref{rem: identity relating delta_i and s_m}. This implies that $c_m = h + H \mathsf{n}(r)^{\delta_m}$ with $H \geq 1$ and $h\in\llbracket 0,\mathsf{n}(r)^{\delta_m} - 1 \rrbracket$. Using the identity $\mathsf{n}(r)^{\delta_n}r^{s_n} = \mathsf{d}(r)^{\delta_n}r^{s_{n + 1}}$ we obtain a factorization $z_{m + 1}$ from $z_m$ in the following manner:
	\[
	\sum_{i = 0}^{m - 1} c_ir^{s_i} + c_m r^{s_m} = \sum_{i = 0}^{m - 1} c_ir^{s_i} + h r^{s_m} + H\mathsf{d}(r)^{\delta_m} r^{s_{m + 1}}\!\!,
	\]
	where
	\begin{equation} \label{eq: expanding coefficient}
	\begin{split}
	H \mathsf{d}(r)^{\delta_m} & = r^{-\delta_m}(c_m - h) \geq r^{-\delta_m}\left(c_m + 1 - \mathsf{n}(r)^{\delta_m} \right)\\
	& > r^{-\delta_m}\left( 1 - \mathsf{n}(r)^{\delta_m} + \prod_{j = -1}^{m - 1} r^{-\delta_j} \sum_{k = m}^{\infty} \left(\mathsf{n}(r)^{\delta_k} - 1\right)r^{s_k}  \right)\\
	& = \prod_{j = -1}^{m} r^{-\delta_j} \sum_{k = m + 1}^{\infty} \left(\mathsf{n}(r)^{\delta_k} - 1\right)r^{s_k}\!\!.
	\end{split}
	\end{equation}
	The last equality in Equation~\eqref{eq: expanding coefficient} follows from Remark~\ref{rem: identity relating delta_i and s_m}. By induction, we have $|\mathsf{Z}(K)| = \infty$, which implies that $M_{r,S}$ does not satisfy the ACCP by Theorem~\ref{prop: ACCP implies FF}. This contradiction proves that our hypothesis is untenable; hence the series~\ref{eq: series determining ACCP} does not converge. This, in turn, implies that the series $\sum_{k = 0}^{\infty} \mathsf{n}(r)^{\delta_{k}} r^{s_k}$ does not converge either, and our result follows by~\cite[Theorem 3.39]{Rudin}.
\end{proof}
\begin{cor} \label{cor: bounded deltas implies nonACCP}
	Let $M_{r,S}$ be an atomic exponential Puiseux monoid with $r < 1$. If there exists $k\in\nn$ such that $\delta_n < k$ for all $n\in\nn$ then $M_{r,S}$ does not satisfy the ACCP.
\end{cor}
Note that Corollary~\ref{cor: bounded deltas implies nonACCP} is a generalization of~\cite[Corollary 4.4]{CGG2019}. On the other hand, the converse of Proposition~\ref{prop: necessary condition for satisfying the ACCP} does not hold as the following example illustrates.

\begin{example} \label{ex: reverse implication does not hold}
	Let $a,b \in\nn$ such that $1 < a < b$. We start by constructing, iteratively, a sequence of rational numbers $(q_n)_{n \geq 1}$ converging to $\log_a b$ such that $\mathsf{n}(q_n) = \mathsf{d}(q_{n + 1})$ and  $1 < q_n < \log_a b$ for all $n \in\nn$. Let $n_1 \in\nn$ such that $1/n_1 < \log_a b - 1$. It is not hard to see that there exists a rational number $q_1$ such that $\log_a b - 1/n_1 \leq q_1 < \log_a b$ and $\mathsf{d}(q_1) = n_1$. Note that $\mathsf{n}(q_1) > n_1$ since $q_1 > 1$. Now assume that $q_1, \ldots, q_k \in \qq_{>1}$ have been defined such that $\log_a b - 1/\mathsf{d}(q_i) \leq q_i < \log_a b$ for all $i \in \llbracket 1,k \rrbracket$ and $\mathsf{n}(q_i) = \mathsf{d}(q_{i + 1})$  for all $i \in\llbracket 1,k - 1 \rrbracket$. Then there exists a rational number $q_{k + 1}$ such that $\log_a b - 1/\mathsf{n}(q_k) \leq q_{k + 1} < \log_a b$ and $\mathsf{d}(q_{k + 1}) = \mathsf{n}(q_k) < \mathsf{n}(q_{k + 1})$. It is easy to see that the sequence $(q_n)_{n \geq 1}$ satisfies the aforementioned requirements.
	
	Consider the exponential Puiseux monoid $M_{r,S}$ where $r = a/b$ and $\delta_n = \mathsf{d}(q_{n + 1})$ for all $n \in\nn_0$. Clearly, $M_{r,S}$ is atomic by Proposition~\ref{prop:atomicity of exponential PMs}. On the other hand, note that $\delta_{n + 1}/\delta_n = q_{n + 1}$ for all $n \in\nn_0$. Since $\delta_{n + 1}/\delta_n < \log_{\mathsf{n}(r)} \mathsf{d}(r)$ for all $n \in\nn_0$, it is not hard to see that $\mathsf{d}(r)^{\delta_n} > \mathsf{n}(r)^{\delta_{n + 1}}$ for all $n \in\nn_0$, which implies that $M_{r,S}$ does not satisfy the ACCP by virtue of the identity
	\[
		\mathsf{n}(r)^{\delta_n}r^{s_n}= \left(\mathsf{d}(r)^{\delta_n}-\mathsf{n}(r)^{\delta_{n+1}}\right)r^{s_{n+1}} + \mathsf{n}(r)^{\delta_{n+1}}r^{s_{n+1}}\!\!.
	\]
	Let $R = \log_{\mathsf{n}(r)} \mathsf{d}(r)$. Next, we proceed to show (through some cumbersome computations) that $\limsup_{n \to\infty} \delta_n/s_n = R - 1$ holds. We have
	
	\begin{align*}
		\limsup_{n \to\infty} \frac{\delta_n}{s_n} & = \limsup_{n \to\infty} \frac{\delta_n}{\delta_0 + \cdots + \delta_{n-1}} \hspace{.3cm}(\text{by Remark }\ref{rem: identity relating delta_i and s_m})\\
		& = \limsup_{n \to\infty} \frac{\delta_n}{\delta_{n-1}} \left( \frac{1}{\frac{\delta_0}{\delta_{n - 1}} +  \cdots + \frac{\delta_{n - 2}}{\delta_{n - 1}} + 1} \right)\\
		& = \limsup_{n \to\infty} \frac{\delta_n}{\delta_{n-1}} \left( \frac{1}{         \frac{\delta_0}{\delta_{1}}\frac{\delta_1}{\delta_2} \cdots \frac{\delta_{n - 2}}{\delta_{n - 1}}        +  \cdots + \frac{\delta_{n - 2}}{\delta_{n - 1}} + 1} \right)\\
		& \leq \limsup_{n \to\infty} R \left( \frac{1}{ \frac{1}{R^{n - 1}} + \cdots + \frac{1}{R} + 1 } \right) \hspace{.3cm}\left(\text{since }\delta_{n + 1}/\delta_n < R \text{ for all }n \in \nn_0\right)\\
		& = \limsup_{n \to\infty} R \left( \frac{R^{n - 1}}{1 + R + \cdots + R^{n - 1}} \right)\\
		& = \limsup_{n \to\infty} \frac{R^{n + 1} - R^n}{R^n - 1} = R - 1.
%		\begin{split}
%			\limsup_{n \to\infty} \frac{\delta_n}{s_n} & = \limsup_{n \to\infty} \frac{\delta_n}{\delta_0 + \cdots + \delta_{n-1}}\\
%			& = \limsup_{n \to\infty} \frac{\delta_n}{\delta_{n-1}} \left( \frac{1}{\frac{\delta_0}{\delta_{n - 1}} +  \cdots + \frac{\delta_{n - 2}}{\delta_{n - 1}} + 1} \right)\\
%			& = \limsup_{n \to\infty} \frac{\delta_n}{\delta_{n-1}} \left( \frac{1}{         \frac{\delta_0}{\delta_{1}}\frac{\delta_1}{\delta_2} \cdots \frac{\delta_{n - 2}}{\delta_{n - 1}}        +  \cdots + \frac{\delta_{n - 2}}{\delta_{n - 1}} + 1} \right)\\
%			& \leq \limsup_{n \to\infty} R \left( \frac{1}{ \frac{1}{R^{n - 1}} + \cdots + \frac{1}{R} + 1 } \right)\\
%			& = \limsup_{n \to\infty} R \left( \frac{R^{n - 1}}{1 + R + \cdots + R^{n - 1}} \right)\\
%			& = \limsup_{n \to\infty} \frac{R^{n + 1} - R^n}{R^n - 1} = R - 1.
%		\end{split}
	\end{align*}
	
	Now let $\epsilon > 0$. Since the sequence $(\delta_n/\delta_{n + 1})_{n \geq 0}$ converges to $R^{-1}$ and $\delta_n/\delta_{n + 1} > R^{-1}$ for all $n \in\nn_0$, there exists $t \in \nn$ such that $\delta_k/\delta_{k + 1} - R^{-1} < \epsilon/R$ for all $k \in\nn_{\geq t}$. Let $T \in\nn_{\geq t}$ such that $\delta_T/\delta_{T + 1} \geq \delta_k/\delta_{k + 1}$ for all $k \in \nn_{\geq t}$. Thus,
	\begin{align*}
		\limsup_{n \to\infty} \frac{\delta_n}{s_n} & = \limsup_{n \to\infty} \frac{\delta_n}{\delta_0 + \cdots + \delta_{n-1}}\\
		& = \limsup_{n \to\infty} \frac{\delta_n}{\delta_{n-1}} \left( \frac{1}{\frac{\delta_0}{\delta_{n - 1}} +  \cdots + \frac{\delta_{n - 2}}{\delta_{n - 1}} + 1} \right)\\
		& = \limsup_{n \to\infty} \frac{\delta_n}{\delta_{n-1}} \left( \frac{1}{\frac{\delta_t}{\delta_{n - 1}} +  \cdots + \frac{\delta_{n - 2}}{\delta_{n - 1}} + 1} \right)\\
		& = \limsup_{n \to\infty} \frac{\delta_n}{\delta_{n-1}} \left( \frac{1}{         \frac{\delta_t}{\delta_{t + 1}}\frac{\delta_{t + 1}}{\delta_{t + 2}} \cdots \frac{\delta_{n - 2}}{\delta_{n - 1}}        +  \cdots + \frac{\delta_{n - 2}}{\delta_{n - 1}} + 1} \right) \\
		& \geq R \cdot \limsup_{n \to\infty} \left( \frac{1}{ \frac{\delta_T}{\delta_{T + 1}}^{n - t - 1} + \cdots + \frac{\delta_T}{\delta_{T + 1}} + 1 } \right)\\
		& = R \cdot \limsup_{n \to\infty}  \frac{1 - \frac{\delta_T}{\delta_{T + 1}}}{1 - \left(\frac{\delta_T}{\delta_{T + 1}}\right)^{n - t}}
		\end{align*}
		
		\begin{align*}
		& = R \cdot \left(1 - \frac{\delta_T}{\delta_{T + 1}}\right) = R \cdot \left( 1 - \frac{1}{R} + \frac{1}{R} - \frac{\delta_T}{\delta_{T + 1}} \right)\\
		& = R - 1 - R \cdot \left(\frac{\delta_T}{\delta_{T + 1}}  - \frac{1}{R}\right) \geq R - 1 - \epsilon.
	\end{align*}
	Since $\epsilon$ is an arbitrary positive real number, $\limsup_{n \to\infty} \delta_n/s_n \geq R - 1$. Note that $\mathsf{n}(r) \limsup_{n \to\infty} \sqrt[s_n]{\mathsf{n}(r)^{\delta_n}} = \mathsf{n}(r)^R = \mathsf{d}(r)$, which implies that $M_{r,S}$ satisfies Equation~\eqref{eq: limsup}. However, the monoid $M_{r,S}$ does not satisfy the ACCP as was already proved.
\end{example}

 We now provide a sufficient condition for an exponential Puiseux monoid to satisfy the ACCP, but first let us prove a lemma.

\begin{lemma} \label{lem: truncation does not affect the atomic properties}
	Let $M_{r,S}$ be an atomic exponential Puiseux monoid, and let $s_i$ be an element of $S$. Then $M_{r,S}$ satisfies the ACCP if and only if $M_{r, S - s_i}$ satisfies the ACCP.
\end{lemma}
\begin{proof}
	If $r \geq 1$ then our result follows trivially by~\cite[Theorem 5.6]{fG16}. Consequently, one may assume without loss that $r < 1$. Suppose that $M_{r,S - s_i}$ does not satisfy the ACCP, and let $(x_n')_{n \geq 1}$ be a sequence of elements of $M_{r,S - s_i}^{\bullet}$ such that $x_n' = x_{n + 1}' + y_n'$ for some $y_n' \in M_{r,S - s_i}^{\bullet}$ and for all $n \in\nn$. Now set $x_n = r^{s_i}x_n'$ and $y_n = r^{s_i}y_n'$ for all $n\in\nn$. Note that $(x_n)_{n \geq 1}$ is a sequence of elements of $M_{r,S}^{\bullet}$ such that $x_n = x_{n + 1} + y_n$ for some $y_n\in M_{r,S}^{\bullet}$. Then $M_{r,S}$ does not satisfy the ACCP, and the direct implication follows. 
	
	We use a similar argument for the reverse implication: Let $(x_n)_{n \geq 1}$ be a sequence of elements of $M_{r,S}^{\bullet}$ such that $x_n = x_{n + 1} + y_n$ for some $y_n \in M_{r,S}^{\bullet}$ and for all $n \in\nn$. Fix $K \in\nn_0$, and let $C_K = \{m\in\nn : mr^K \,|_{M_{r,S}} \,x_n \text{ for some } \,n\in\nn\}$. Since $(x_n)_{n \geq 1}$ is a bounded sequence, $C_K$ is a finite set. Consequently, there are infinitely many elements of the sequence $(x_n)_{n \geq 1}$ whose factorization of minimum length share the same coefficient for the atom $r^K$\!. This, in turn, implies that there is no loss in assuming that $r^K$ does not show up in the factorization of minimum length of $x_n$ for any $n \in\nn$. By the same token, we may assume that, for each $n \in\nn$, $r^{s_m}$ does not show up in the factorization of minimum length of $x_n$ for any $m\in\llbracket 0,i \rrbracket$. Now set $x_n' = r^{-s_i}x_n$ and $y_n' = r^{-s_i}y_n$ for all $n\in\nn$. It is not hard to see that $(x_n')_{n \geq 1}$ is a sequence of elements of $M_{r,S - s_i}^{\bullet}$ such that $x_n' = x_{n + 1}' + y_n'$ for some $y_n'\in M_{r,S - s_i}^{\bullet}$ and our proof concludes.
	%		Next, we proceed to prove $(2)$. If $M_{r,S - s_i}$ is not a BF-monoid, then there exists $x \in M_{r,S - s_i}$ such that $\mathsf{L}(x)$ is unbounded. Consequently, $\mathsf{L}(r^{s_i}x)$ is unbounded, where $r^{s_i}x \in M_{r,S}$. As for the reverse implication, if $M_{r,S}$ is not a BF-monoid, then there exists $x\in M_{r,S}$ such that $\mathsf{L}(x)$ is unbounded. Consider the set
	%		\[
	%			S = \left\{(c_0, \ldots, c_i) : \sum_{j = 0}^{i}c_jr^{s_j} \, |_{M_{r,S}} \,x \text{ and } c_j \nmid x - \left(\sum_{j = 0}^{i}c_jr^{s_j}\right) \text{ for }j\in\llbracket 0,i \rrbracket\right\}.
	%		\] 
	%		Clearly, $S$ is a finite set, which implies that there is no loss in assuming that $r^{s_l} \nmid x$ for $l\in\llbracket 0,i \rrbracket$. It is not hard to see that $r^{-s_i}x \in M_{r, S-s_i}$ and $\mathsf{L}(r^{-s_ix})$ is unbounded. Then the reverse implication of $(2)$ follows.
	%		Finally, the proof of $(3)$ follows exactly the same pattern as the proof of $(2)$; we leave it to the reader.
\end{proof}

\begin{prop} \label{prop: sufficient condition to satisfy the ACCP}
	Let $M_{r,S}$ be an atomic exponential Puiseux monoid with $r < 1$. If there exists $m\in\nn_0$ such that $\mathsf{d}(r)^{\delta_n} < \mathsf{n}(r)^{\delta_{n+1}}$ for $n \geq m$ then $M_{r,S}$ satisfies the ACCP.
\end{prop}

\begin{proof}
	By virtue of Lemma~\ref{lem: truncation does not affect the atomic properties}, one may assume that $\mathsf{d}(r)^{\delta_n} < \mathsf{n}(r)^{\delta_{n + 1}}$ for all $n \in\nn_0$. By way of contradiction, suppose that $M_{r,S}$ does not satisfy the ACCP. Then there exists a sequence $(x_n)_{n \geq 1}$ of elements of $M_{r,S}^{\bullet}$ such that $x_n = x_{n + 1} + y_n$ for some $y_n \in M_{r,S}^{\bullet}$ and for all $n \in\nn$. Let $z_n = \sum^{m_n}_{i=0} c_{n,i}\, r^{s_{i}}$ with $m_n, c_{n,i} \in\nn_0$ for each $i \in\llbracket 0,m_n \rrbracket$ be the factorization of minimum length of $x_n$ for each $n \in\nn$. Take an arbitrary $n \in\nn$. Clearly, $x_n$ has no factorization of maximum length which implies that there exists $j \in\llbracket 0,m_n \rrbracket$ such that $c_{n,j} \geq \mathsf{n}(r)^{\delta_j}$ by Lemma~\ref{Lemma 2.2} (part $(3)$). Then $j = 0$; otherwise, $z_n$ is not the factorization of minimum length of $x_n$ given that $\mathsf{n}(r)^{\delta_j} > \mathsf{d}(r)^{\delta_{j - 1}}$. Without loss of generality, we may assume that there exists $i_n \in \llbracket 0,m_{n+1} \rrbracket$ such that $c_{n,i_n} = 0$ and $c_{n+1,i_n} \neq 0$.
	
	We shall prove that $c_{n,0} > c_{n+1,0}$. Take $z \in\mathsf{Z}(y_n)$ and let $z'_1 = z_{n+1} + z \in\mathsf{Z}(x_n)$. By applying the identity $\mathsf{d}(r)^{\delta_k}r^{s_{k + 1}} = \mathsf{n}(r)^{\delta_k}r^{s_k}$ with $k \in \nn_0$ finitely many times, one can obtain a sequence $z'_1, z'_2, \ldots, z'_l$ of distinct factorizations of $x_n$ such that $|z'_i| > |z'_{i + 1}|$ for all $i \in\llbracket 1, l - 1 \rrbracket$ and $z'_l = z_n$. Since $c_{n,i_n} = 0$ and $c_{n+1,i_n} \neq 0$ for some $i_n \in \llbracket 0,l_{n+1} \rrbracket$, we have $z'_1 \neq z'_l$. Note that, due to the nature of the transformation $\mathsf{d}(r)^{\delta_k}r^{s_{k + 1}} = \mathsf{n}(r)^{\delta_k}r^{s_k}$, where $k \in \nn_0$, the coefficient of $r^0$ in $z'_i$ is less than or equal to the corresponding coefficient in $z'_{i + 1}$ for all $i \in\llbracket 1, l - 1 \rrbracket$. Since $c_{n,i} < \mathsf{d}(r)^{\delta_{i - 1}} < \mathsf{n}(r)^{\delta_i}$ for each $i \in\llbracket 1,m_n \rrbracket$ by Lemma~\ref{Lemma 2.2}, we have that $c_{n,0}$ is strictly bigger than the coefficient of $r^0$ in $z'_{l - 1}$ which, in turn, implies that $c_{n,0} > c_{n + 1,0}$. This is a contradiction as there is no infinite strictly decreasing sequence of positive integers. Hence $M_{r,S}$ satisfies the ACCP.  
\end{proof}

\begin{remark} \label{rem: second necessary condition to satisfy the ACCP}
	Note that if there exists $m\in\nn$ such that $\mathsf{d}(r)^{\delta_n} > \mathsf{n}(r)^{\delta_{n+1}}$ for $n \geq m$ then $M_{r,S}$ does not satisfy the ACCP since 
	\[
		\mathsf{n}(r)^{\delta_n}r^{s_n}= \left(\mathsf{d}(r)^{\delta_n}-\mathsf{n}(r)^{\delta_{n+1}}\right)r^{s_{n+1}} + \mathsf{n}(r)^{\delta_{n+1}}r^{s_{n+1}}
	\]
	as we pointed out in Example~\ref{ex: reverse implication does not hold}.
\end{remark}

Using Proposition~\ref{prop: sufficient condition to satisfy the ACCP} we can easily construct many different examples of exponential Puiseux monoids satisfying the ACCP. On the other hand, Remark~\ref{rem: second necessary condition to satisfy the ACCP} provides us with a large number of exponential Puiseux monoids not satisfying the ACCP. Consider the following example.

\begin{example}
	Let $p(x)$ be a non-constant polynomial with positive integer coefficients, and consider the exponential Puiseux monoid $M_{r,S}$ with $r \in \qq_{< 1}$ and $\delta_n = p(n)$ for all $n \in\nn_0$. Note that $\mathsf{d}(r)^{\delta_n} > \mathsf{n}(r)^{\delta_{n + 1}}$ if and only if $\frac{\ln \mathsf{d}(r)}{\ln \mathsf{n}(r)} > \frac{\delta_{n + 1}}{\delta_n}$. Since
	\[
		\lim_{n \to\infty} \frac{\delta_{n + 1}}{\delta_n} = \lim_{n \to\infty} \frac{p(n + 1)}{p(n)} = 1,
	\]
	there exists $N \in\nn$ such that $\frac{\ln \mathsf{d}(r)}{\ln \mathsf{n}(r)} > \frac{\delta_{n + 1}}{\delta_n}$ for all $n \geq N$. Then $M_{r,S}$ does not satisfy the ACCP.
\end{example}

\section{Exponential Puiseux Semirings}

Exponential Puiseux monoids are generalizations of rational cyclic Puiseux monoids which are closed under multiplication and thus semirings. Motivated by this, we study the exponential Puiseux monoids that are also closed under the standard multiplication of $\qq$. Throughout this section, we allow exponential Puiseux monoids to be finitely generated.

A \emph{commutative semiring} $S$ is a nonempty set endowed with two binary operations (called addition and multiplication and denoted by + and $\cdot$ respectively) satisfying the following conditions:
\begin{enumerate}
	\item $(S,+)$ is a monoid with identity element $0$;
	\item $(S, \cdot)$ is a commutative semigroup with identity element $1$;
	\item for $a,b,c \in S$ we have $(a + b)\cdot c = a\cdot c + b\cdot c$;
	\item $0 \cdot a = 0$ for all $a \in S$.
\end{enumerate}

The more general definition of a `semiring' does not assume that the semigroup $(S, \cdot)$ is commutative; however, this more general type of algebraic object is not of interest for us here. Accordingly, from now on we use the single term \emph{semiring}, implicitly assuming commutativity. For extensive background information on semirings, we refer readers to the monograph~\cite{JG1999} of J. S. Golan.
\smallskip

During the last decade, semirings have received some attention in the context of factorization theory. For instance, in~\cite{CCMS09} the authors investigated, for quadratic algebraic integers $\tau$, the elasticity of the multiplicative structure of the semiring $\nn_0[\tau]$ and in~\cite{vP15}, factorization aspects of semigroup semirings were studied. More recently, a systematic investigation of the factorizations of the multiplicative structure of the semiring $\nn_0[x]$ was provided in \cite{CF19}. Finally and most relevant to this work, the atomicity of the additive and multiplicative structures of rational cyclic semirings were considered in~\cite{CGG19} and~\cite[Section~3]{NBFG2020}, respectively.

We now show that the subsets $N$ of $\nn_0$ for which an exponential Puiseux monoid $M_{r,N}$ is a semiring are precisely the numerical monoids.

\begin{prop} \label{prop: characterization of exponential Puiseux semirings}
	Take $r \in \qq_{> 0} \setminus \nn$ such that $\mathsf{n}(r) \neq 1$ and let $N$ be a nonempty subset of $\nn_0$. Consider the Puiseux monoid $(S_{r,N},+) = \langle r^n \mid n \in N \rangle$. Then $N$ is isomorphic to a numerical monoid if and only if $S_{r,N}$ is a semiring.
\end{prop}

\begin{proof}
	First, we verify that $(S_{r,N},+)$ is atomic with $\mathcal{A}(S_{r,N},+) = \{r^n \mid n \in N\}$. It follows from Proposition~\ref{prop:atomicity of exponential PMs} that $\mathcal{A}(S_r,+) = \{r^n \mid n \in \nn_0\}$, where $(S_r,+)$ is the rational cyclic semiring parameterized by $r$. Because the inclusion $\mathcal{A}(S_r,+) \cap S_{r,N} \subseteq \mathcal{A}(S_{r,N},+)$ holds, $r^n \in \mathcal{A}(S_{r,N},+)$ for every $n \in N$. This, in turn, implies that $(S_{r,N},+)$ is atomic.
	
	To argue the direct implication suppose that $N$ is isomorphic to a numerical monoid. Since $N$ is closed under addition, the set $\{r^n \mid n \in N\}$ is closed under multiplication. This immediately implies that $(S_{r,N},+)$ is multiplicatively closed. As $0 \in N$, the multiplicative semigroup $(S^{\bullet}_{r,N},\cdot)$ has an identity. Hence $S_{r,N}$ is a semiring.
	
	To prove the reverse implication suppose that $S_{r,N}$ is a semiring. Take $n,m \in N$. So we have $r^{n+m} = r^n r^m \in S_{r,N}$. Since $r^{n + m} \in \mathcal{A}(S_r,+) \cap S_{r,N}$, we also have that $r^{n + m} \in \mathcal{A}(S_{r,N},+)$.
%	 We have already seen that $\mathcal{A}(S_{r,N}) = \{r^n \mid n \in N\}$. So we can write
%	\begin{equation} \label{eq:sum decomposition numerical rational semiring}
%	r^{n+m} = \sum_{i=1}^k c_i r^{n_i}
%	\end{equation}
%	for some index $k \in \nn$, coefficients $c_1, \dots, c_k \in \nn$, and exponents $n_1, \dots, n_k \in N$. One can assume without loss that $n_1 < n_2 < \dots < n_k$.
%	
%	\noindent {\it Case 1}: $r > 1$. In this case, it follows from Equation~(\ref{eq:sum decomposition numerical rational semiring}) that $n+m \ge n_k$. Then one notices that the left-hand side of $\mathsf{n}(r)^{n+m} = \sum_{i=1}^k c_i \mathsf{n}(r)^{n_i} \mathsf{d}(r)^{n+m - n_i}$ is not divisible by $\mathsf{d}(r)$. This enforces the equality $n+m=n_k$. As a result, $k=1$ and $c_1=1$, which imply that $r^{n+m} \in \mathcal{A}(S_{r,N})$. 
%	
%	\noindent {\it Case 2}: $r < 1$. Now Equation~(\ref{eq:sum decomposition numerical rational semiring}) implies that $n+m \le n_1$. Then one observes that every summand in the right-hand side of $\mathsf{n}(r)^{n+m} \mathsf{d}(r)^{n_k - n - m} = \sum_{i=1}^k c_i \mathsf{n}(r)^{n_i} \mathsf{d}(r)^{n_k - n_i}$ is divisible by $\mathsf{n}(r)^{n_1}$. As a result, $n+m = n_1$, forcing the equalities $k=1$ and $c_1=1$. As in the previous case, one obtains that $r^{n+m} \in \mathcal{A}(S_{r,N})$.
This guarantees that $n+m \in N$. Hence $N$ is closed under addition. Because $S_{r,N}$ is a semiring, $1 \in S_{r,N}$. So we can write
	\begin{equation} \label{eq:sum decomposition numerical rational semiring}
	1 = \sum_{i=1}^k c_i r^{n_i}
	\end{equation}
for some index $k \in \nn$, coefficients $c_1, \dots, c_k \in \nn$, and exponents $n_1, \dots, n_k \in N$. One can assume without loss that $n_1 < n_2 < \dots < n_k$. After cleaning denominators, we obtain $\mathsf{d}(r)^{n_k} = \sum_{i=1}^k c_i \mathsf{n}(r)^{n_i} \mathsf{d}(r)^{n_k - n_i}$, whose left-hand side is not divisible by $\mathsf{n}(r)$. As a result $n_1 = 0$, which implies that $0 \in N$. Hence $N$ is isomorphic to a numerical monoid.
\end{proof}

Before continuing, let us make a definition to avoid long descriptions. 

\begin{definition}
	Let $r \in \qq_{>0}$ and $N \subseteq \nn_0$. We say that $(S_{r,N},+) = \langle r^n \mid n \in N \rangle$ is an \emph{exponential Puiseux semiring} provided that $N$ is a numerical monoid.
\end{definition}

Next we prove that whether or not the ACCP, the BFP, and the FFP hold for the multiplicative monoid of an exponential Puiseux semiring can be completely determined by whether or not the same property holds for the corresponding rational cyclic semiring. First let us collect some lemmas.

\begin{lemma} \label{lemma: the multiplicative monoid of an exponential Puiseux semiring is reduced}
	Let $S_{r,N}$ be an exponential Puiseux semiring with $\mathsf{n}(r) \neq 1$. Then $(S_{r,N}^{\bullet},\cdot)$ is a monoid.
\end{lemma}

\begin{proof}
	It is clear that the commutative semigroup $(S_{r,N}^\bullet,\cdot)$ is cancellative. So it suffices to verify that it is reduced. By way of contradiction, suppose that $(S_{r,N}^{\bullet},\cdot)$ is not reduced. So there exists $q \in S_{r,N}^{\bullet}\setminus \{1\}$ such that $q^{-1}$ is also an element of $S_{r,N}^{\bullet}$. It follows from Proposition~\ref{prop:atomicity of exponential PMs} that either $q$ or $q^{-1}$ is not an additive atom of the exponential Puiseux monoid $(S_{r,N},+)$. This implies that $1$ is not an atom of $(S_{r,N},+)$. Indeed if, for example, $q^{-1} = x+ y$ then $1 = qq^{-1} = qx + qy$. Then the exponential Puiseux monoid $(S_{r,N},+)$ is antimatter. However, this contradicts Proposition~\ref{prop:atomicity of exponential PMs}.
\end{proof}

%The following result is well known, but for the convenience of the reader we provide a proof here.
%
%\begin{lemma} \label{lemma: a submonoid of a monoid that satisfies the ACCP also satisfies the ACCP}
%	Let $M$ be a monoid that satisfies the ACCP and $M'$ a submonoid of $M$. Then $M'$ also satisfies the ACCP.
%\end{lemma}
%
%\begin{proof}
%	Assume that $M'$ does not satisfy the ACCP. Then there exists a sequence $(x_n)_{n\geq1}$ of different elements of $M'$ such that $x_{n + 1} \,|_{M'}\, x_n$ for all $n \in\nn$. Since $M'$ is a submonoid of $M$, we have that $(x_n)_{n\geq1}$ is a sequence of different elements of $M$ such that $x_{n + 1} \,|_M\, x_n$ for all $n \in\nn$. This implies that $M$ does not satisfy the ACCP, which concludes the proof.
%\end{proof}

\begin{lemma} \label{lemma: finite number of divisor of the form r^n}
	Let $S_{r,N}$ be an exponential Puiseux semiring with $r < 1 < \mathsf{n}(r)$. Then for each $x \in S_{r,N}^{\bullet}$ there exists $m \in\nn$ such that $r^n \nmid_{(S_{r,N}^{\bullet},\cdot)} x$ for any $n \in \nn_{\geq m}$.
\end{lemma}

\begin{proof}
	Since $r < 1 < \mathsf{n}(r)$, the monoid $(S_{r,N},+)$ is atomic by Proposition~\ref{prop:atomicity of exponential PMs}. Suppose, by way of contradiction, there exists $x \in S_{r,N}^{\bullet}$ for which the set $N' = \{n \in \nn : r^{n} \,|_{(S_{r,N}^{\bullet},\cdot)}\, x\}$ has infinite cardinality. Then, for each $n \in N'$, there exists an additive factorization $\sum_{i = n}^{m_n} d_i r^{i} \in\mathsf{Z}_{(S_{r,N},+)}(x)$ for some index $m_n \in\nn_{\geq n}$, coefficients $d_n, \ldots, d_{m_n} \in \nn_0$, and exponents $n, \ldots, m_n \in\nn$. This implies that $\mathsf{n}(r)^{n} \,|\, \mathsf{n}(x)$ for all $n \in N'$, which is a contradiction given that $N'$ has infinite cardinality. This contradiction proves that our hypothesis is untenable, and our result follows.
\end{proof}

Now we are in a position to prove the main result of this section.

\begin{theorem} \label{theorem: atomic structure of exponential Puiseux semirings is atomic structure of rational cyclic semiring}
	Let $S_{r,N}$ be an exponential Puiseux semiring, and let $S_r$ be the corresponding rational cyclic semiring. The following statements hold.
	\begin{enumerate}
%		\item $\left(S_{r,N}^{\bullet},\cdot\right)$ is atomic if and only if $\left(S_r^{\bullet},\cdot\right)$ is atomic.
		\item $(S_{r,N}^{\bullet},\cdot)$ satisfies the ACCP if and only if $(S_r^{\bullet},\cdot)$ satisfies the ACCP.
		\item $(S_{r,N}^{\bullet},\cdot)$ satisfies the BFP if and only if $(S_r^{\bullet},\cdot)$ satisfies the BFP.
		\item $(S_{r,N}^{\bullet},\cdot)$ satisfies the FFP if and only if $(S_r^{\bullet},\cdot)$ satisfies the FFP.
	\end{enumerate}
\end{theorem}

\begin{proof}
	If either $\mathsf{n}(r) = 1$ or $\mathsf{d}(r) = 1$ then $S_{r,N}^{\bullet} = S_r^{\bullet}$, so there is no loss in assuming that $\mathsf{n}(r) \neq 1$ and $r \not\in\nn$. %In virtue of Lemma~\ref{lemma: the multiplicative monoid of an exponential Puiseux semiring is reduced}, $(S_{r,N}^{\bullet},\cdot)$ and $(S_r^{\bullet},\cdot)$ are both reduced. 
	The reverse implication of $(1)$ follows from the fact that submonoids of monoids satisfying the ACCP also satisfy the ACCP, while the reverse implications of $(2)$ and $(3)$ follow for \cite[Corollary 1.3.3]{GH06b} and \cite[Corollary 1.5.7]{GH06b}, respectively. On the other hand, if $r > 1$ then $(S_r^{\bullet},\cdot)$ is atomic by~\cite[Proposition 3.11]{NBFG2020}, and since the atoms of $(S_r^{\bullet},\cdot)$ are strictly bigger than $1$, $(S_r^{\bullet},\cdot)$ is an FFM. Consequently, we may assume without loss of generality that $r < 1 < \mathsf{n}(r)$. 
	
	To prove the direct implication of $(1)$, suppose that $(S_r^{\bullet},\cdot)$ does not satisfy the ACCP. Then there exist sequences $(x_m)_{m \geq 1}$ and $(y_m)_{m \geq 1}$ of elements of $S_r^{\bullet}\setminus\{1\}$ such that $x_m = x_{m + 1}\cdot y_m$ for all $m \in\nn$. First note that by Lemma~\ref{lemma: finite number of divisor of the form r^n} there is $k \in \nn$ so that $y_m > 1$ for each $m > k$. Thus we may assume that $y_m > 1$ for each $m \in \nn$. Also note that the sequence $(y_m)_{m \geq 1}$ converges to $1$. Indeed, if there exists a rational number $q > 1$ such that $y_m > q$ for infinitely many indices $m$ then it is not hard to see that $0$ is a limit point of the divisors of $x_1$ in $S_r^{\bullet}$, contradicting Lemma~\ref{lemma: finite number of divisor of the form r^n}. Finally, since the sequence $(y_m)_{m \geq 1}$ converges to $1$, there exists $h \in\nn$ such that $y_m \in S_{r,N}^{\bullet}$ for $m > h$. Now consider the sequence $(r^{F(N) + 1} \,x_m)_{m \geq 1}$, where $F(N)$ is the Frobenius number of the numerical monoid $N$. Clearly, $r^{F(N) + 1}\, x_m \in S_{r,N}^{\bullet}$ for all $m \in\nn$. Moreover, $r^{F(N) + 1} x_m = (r^{F(N) + 1} x_{m + 1})\cdot y_m$ for each $m \in\nn$, and since $y_m \in S_{r,N}^{\bullet}$ for $m > h$, it follows that $(S_{r,N}^{\bullet},\cdot)$ does not satisfy the ACCP, from which $(1)$ follows.
	
	To prove the direct implications of $(2)$ and $(3)$, note first that $(S_r^{\bullet},\cdot)$ is atomic. Indeed, since $(S_{r,N}^{\bullet},\cdot)$ is a BFM, it satisfies the ACCP. Then, by $(1)$, the monoid $(S_r^{\bullet},\cdot)$ also satisfies the ACCP, so it is atomic. Moreover, $r \in \mathcal{A}(S^{\bullet}_r,\cdot)$. Now let $x \in S_r^{\bullet}\setminus\{1\}$ and consider a factorization $z = a_1 \cdots a_n \in\mathsf{Z}_{(S_r^{\bullet},\cdot)}(x)$. By Lemma~\ref{lemma: finite number of divisor of the form r^n}, there exists a positive integer $m$, which does not depend on $z$, such that $|\{a_i : a_i = r\}| < m$; on the other hand, 
	\[
	\left|\left\{a_i \in S^{\bullet}_r \setminus S^{\bullet}_{r,N} \mid a_i \neq r \right\} \right| < \left\lceil\log_d \left(xr^{-m}\right)\right\rceil,
	\]
	where $d = 1 + r^{F(N)}$\!. In other words, a factorization $z \in \mathsf{Z}_{(S_r^{\bullet},\cdot)}(x)$ contains as factors at most  $h = m + \lceil\log_d (xr^{-m})\rceil$ atoms of $(S_r^{\bullet},\cdot)$ that are not elements of $S_{r,N}^{\bullet}$. Consequently, $y = r^{h(F(N) + 1)} x$ is an element of $S_{r,N}^{\bullet}$. It is not hard to see that if $\mathsf{L}_{(S_r^{\bullet}, \cdot)}(x)$ is unbounded then $\mathsf{L}_{(S_{r,N}^{\bullet},\cdot)}(y)$ is also unbounded, which concludes the proof of $(2)$. As for the direct implication of $(3)$, if $|\mathsf{Z}_{(S_r^{\bullet},\cdot)}(x)| = \infty$ then there are infinitely many atoms of $(S^{\bullet}_r,\cdot)$ dividing $x$ by ~\cite[Proposition 1.5.5]{GH06b}. Since $r$ is the only element of $\mathcal{A}(S^{\bullet}_r, \cdot)_{<1}$, the set $A(x) = \{a \in \mathcal{A}(S_r^{\bullet}, \cdot)_{> 1} : a \,|_{(S_r^{\bullet},\cdot)}\, x\}$ has infinite cardinality. This, in turn, implies that $y$ has infinitely many divisors in $(S_{r,N}^{\bullet},\cdot)$ as $r^{(F(N) + 1)}\cdot a$ is a divisor of $y$ in $(S_{r,N}^{\bullet},\cdot)$ for all $a \in A(x)$. Then $(3)$ follows from~\cite[Proposition 1.5.5]{GH06b}. 

\end{proof}
\begin{cor}
	Let $(S_{r,N},+)$ be an exponential Puiseux monoid with $\mathsf{n}(r) > 1$. Then $(S_{r,N}^{\bullet},\cdot)$ satisfies the ACCP provided that $\mathsf{d}(r) = p^k$ for some $p \in\mathbb{P}$ and $k \in\nn$.
\end{cor}
\begin{proof}
	If $r > 1$ then our result follows from~\cite[Proposition 3.11]{NBFG2020} and Theorem~\ref{theorem: atomic structure of exponential Puiseux semirings is atomic structure of rational cyclic semiring}. Consequently, we may assume that $r < 1$, which implies that $(S_{r},+)$ is atomic by Proposition~\ref{prop:atomicity of exponential PMs}. By way of contradiction, suppose that $(S_r^{\bullet},\cdot)$ does not satisfy the ACCP. As we have already established in the proof of Theorem~\ref{theorem: atomic structure of exponential Puiseux semirings is atomic structure of rational cyclic semiring}, there exist sequences $(x_m)_{m \geq 1}$ and $(y_m)_{m \geq 1}$ of elements of $S_{r}^{\bullet}\setminus\{1\}$ such that $y_m > 1$, $x_m = x_{m + 1}\cdot y_m$ for all $m\in\nn$, and $(y_m)_{m \ge 1}$ converges to $1$. Then $x_1y_1^{-1} \cdots y_n^{-1} \in S_r^{\bullet}$ for all $n \in\nn$. For each $j \in \nn$, let $z_j = \sum_{i = 0}^{n_j} c_{j,i}r^{i}$ be the factorization of minimum length of $y_j$ in $(S_r,+)$ with $n_j, c_{j,i} \in\nn_0$. By virtue of Lemma~\ref{lemma: finite number of divisor of the form r^n}, one can assume that $c_{j,0} \neq 0$ for all $j \in\nn$. Thus,
	\[
		y_j = \frac{c_{j,0}\mathsf{d}(r)^{n_j} + \sum_{i = 1}^{n_j} c_{j,i}\mathsf{n}(r)^{i}\mathsf{d}(r)^{n_j - i}}{\mathsf{d}(r)^{n_j}} = \frac{h \cdot\mathsf{n}(y_j)}{p^{kn_j}}
	\]
	for some $h \in \nn$. It is easy to see that if $p \,|\, h\,\mathsf{n}(y_j)$ then $p \,|\, c_{j,n_j}$. Since $c_{j,n_j} < \mathsf{d}(r)$ by Lemma~\ref{Lemma 2.2}, there exists a prime number $q$ such that $q \,|\, h\,\mathsf{n}(y_j)$ and $q \,\nmid\, \mathsf{d}(r)$. Consequently, $\mathsf{n}(x_1)$ is divisible by infinitely many (counting repetitions) prime numbers, a contradiction. Therefore, $(S_{r,N}^{\bullet},\cdot)$ satisfies the ACCP by Theorem~\ref{theorem: atomic structure of exponential Puiseux semirings is atomic structure of rational cyclic semiring}.
\end{proof}

\section*{Acknowledgments}
	 The authors want to thank Felix Gotti for his mentorship and guidance during the preparation of this paper, and anonymous referees whose careful revision improved the final version. While working on this manuscript, the third author was supported by the University of Florida Mathematics Department Fellowship.

\bigskip

\end{document}